\documentclass[11pt]{amsart}
\usepackage[normalem]{ulem}
\usepackage[usenames]{color}
\usepackage{graphicx,amscd,psfrag,amssymb}

 \newcommand{\R}{\mathbb R}
 \newcommand{\Z}{\mathbb Z}\newcommand{\nn}{\mathbb N}
\newcommand{\ep}{\varepsilon}
\newcommand{\blocks}[1]{\mathcal{B}_{#1}}
\DeclareMathOperator{\Per}{Per}
\DeclareMathOperator{\Int}{Int}

\theoremstyle{plain} \newtheorem{thm}{Theorem}
\newtheorem{cor}[thm]{Corollary} \newtheorem{prop}[thm]{Proposition}

\theoremstyle{definition} \newtheorem{defn}[thm]{Definition}

\newtheorem{ex}[thm]{Example} 

\theoremstyle{remark} 
\newcommand{\intgraph}{\mathcal{I}}
\newcommand{\transgraph}{\mathcal{T}}
\newcommand{\graph}{\mathcal{G}}
\DeclareMathOperator{\ind}{ind}

\newcommand{\ent}{h}
\newcommand{\olent}[1]{\ent_{#1}}
\newcommand{\olenti}{\olent{\intgraph}}
\newcommand{\sep}[1]{\mathcal{B}_{#1}^{\text{sep}}}

\begin{document}

\title{Entropy for symbolic dynamics with overlapping alphabets}

\author{Fabio Drucker}   \author{David Richeson}   \author{Jim Wiseman} \address{Dickinson College\\ Carlisle, PA 17013} \email{druckerf@dickinson.edu}  \address{Dickinson College\\ Carlisle, PA 17013} \email{richesod@dickinson.edu} \address{Agnes Scott College \\ Decatur, GA 30030} \email{jwiseman@agnesscott.edu}


\begin{abstract}
We consider shift spaces in which elements of the alphabet may overlap nontransitively.  We define a notion of entropy for such spaces, give several techniques for computing lower bounds for it, and show that it is equal to a limit of entropies of (standard) full shifts. When a shift space with overlaps arises as a model for a discrete dynamical system with a finite set of overlapping neighborhoods, the entropy gives a lower bound for the topological entropy of the dynamical system.
\end{abstract}

\maketitle

\section{Introduction}

There is a long history of using symbolic dynamics to model more complicated dynamical systems. Suppose, for example, that we want to understand the dynamics of a continuous dynamical system $f:X\to X$ on a compact metric space $X$.  If we have compact sets $N_{1},\ldots,N_{n}\subset X$ and  an $f$-orbit $(x, f(x), f^2(x),\ldots)$ that remains in $\bigcup_i N_i$, then there is a symbol sequence $(i_0,i_1,\ldots)\in\{1,\ldots,n\}^{\nn}$ called an \emph{itinerary} such that  $f^j(x) \in N_{i_j}$ for all $j\ge 0$ (the itinerary may not be unique). We can then use the properties of the set of itineraries (a shift space) to tell us about the dynamics of $f$.

One example of this approach is a Markov partition, in which the sets  $N_{1},\ldots,N_{n}\subset X$ cover the entire space, overlap only on their boundaries, and map across each other in topologically simple ways (\cite[\S9.6]{Ro}).  In this case the resulting shift space of itineraries is a subshift of finite type.  Another example is when the $N_i$'s are pairwise disjoint Conley index pairs , resulting in a cocyclic subshift (\cite{szymczak-decomp,KwaCocyclic,KwaTransfer}).

While these approaches have been very fruitful, there are limitations with each of them.  Markov partitions can be very difficult to construct in practice, and the Conley index pairs must be disjoint, making it impossible to cover a connected space.

In this article we allow the $N_{i}$ to have nontrivial intersections. The complication that arises if we allow the $N_i$'s to intersect is that an orbit may have multiple itineraries. For example, a fixed point in $N_{1}\cap N_{2}$ could have itineraries $(1,1,1,\ldots)$, or $(1,2,1,2,\ldots)$, or any sequence of 1's and 2's. Our task is to extract useful dynamical information about $f$ from the shift space while being careful about the nonuniqueness of itineraries. 

We begin with the definition of our object of study---a shift space with overlaps.

\begin{defn}
A \emph{shift space with overlaps} is a pair $(\Sigma,\intgraph)$ in which $\Sigma$ is a one-sided shift space on the alphabet $\{1,\ldots,n\}$ and $\intgraph$ is a simple graph with vertex set $\{1,\ldots,n\}$. We call $\intgraph$ an \emph{intersection graph}. Two words  $(a_{0},a_{1},\ldots),(b_{0},b_{1},\ldots)\in\Sigma$ (finite or infinite) are \emph{indistinguishable} iff there is an $\intgraph$-edge between the vertices $a_{i}$ and $b_{i}$ for all $i$.
\end{defn}

Shift spaces with overlaps are interesting objects to study on their own, but most often they are a model for an existing dynamical system. Thus we have the following definition.

\begin{defn}
A compact metric space $X$, a continuous map $f:X\to X$, and a collection of nonempty compact sets $N_1,\dots,N_n\subset X$ is called a \emph{dynamical realization} of $(\Sigma,\intgraph)$ provided:

\begin{enumerate}
\item
If $N_{i}\cap N_{j}\ne\emptyset$ and $i\ne j$, then there is an edge from $i$ to $j$ in $\intgraph$.
\item
If $(a_0,a_1,\dots)$ is a (finite or infinite) word in $\Sigma$, then there is a point $x\in X$ such that $f^i(x)\in \Int(N_{a_i})$ for all $i$.
\end{enumerate}
\end{defn}

In practice, we use topological methods to verify that (2) holds. The methods may be elementary, such as boxes that stretch across other boxes, as in the case of Markov partitions or their topological generalizations shown in Figure \ref{fig:stretchacross}. Or they may be more sophisticated; for example, in  \cite{RW7}  we use techniques from Conley index theory to verify the property.

\begin{figure}[ht]
\includegraphics[width=4.5in]{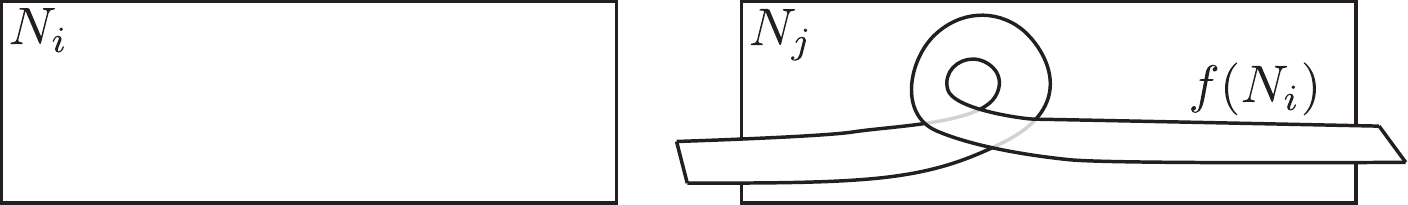}
\caption{}\label{fig:stretchacross}
\end{figure}

For much of this article we assume that $\Sigma$ is a subshift of finite type. When $\Sigma$ is the vertex shift associated to a transition graph $\transgraph$ we superimpose the graphs $\transgraph$ and $\intgraph$ with the edges of $\transgraph$ being solid arrows and the edges of $\intgraph$ dashed segments (see Figure \ref{fig:bigexample}, for example); we will refer to this merged graph as $\transgraph\intgraph$.

We are using $(\Sigma,\intgraph)$ as an abstract model of our dynamical system $f:X\to X$. As such, two indistinguishable elements in $\Sigma$ could be itineraries of the same point in $X$. Thus in a sense we must treat two indistinguishable elements in $\Sigma$ as the same point. However, what makes this scenario interesting is that indistinguishability is not a transitive property, and is thus not an equivalence relation. In general, $(\Sigma,\intgraph)$ is not a topological space.  (It is a tolerance space---see, for example, \cite{soss}.)

Our goal is to extract information about the dynamical complexity of $f$ from $(\Sigma, \intgraph)$, and most often in this article, from the graph $\transgraph\intgraph$.  
We define a notion of topological entropy for shifts with overlap, $\olenti(\Sigma)$, in Section \ref{sec:entropyshiftwithoverlap} and show that it is a lower bound for the entropy of a dynamical realization $f$.  We estimate $\olenti(\Sigma)$ using only disjoint sets in Section~\ref{sec:indep}, discuss decompositions of the shift space in Section~\ref{sec:decomp}, and estimate $\olenti(\Sigma)$ using connected components in Section~\ref{sec:sofic}.  In Section~\ref{sec:higher} we discuss higher shifts, and in Section~\ref{sec:entropytheorem} we prove that $\olenti(\Sigma)$ can be computed as a limit of the entropies of ordinary (nonoverlapping) embedded shifts.

One of the motivations for undertaking this study is its potential application to computational dynamical systems. There is an active research program of using the Conley index and computational topology to give rigorous entropy bounds for dynamical systems. In particular,  one could use rigorous computation to obtain an index system (as described in \cite{RW7}) and from that a shift space with overlaps. So it is important to find techniques for computing the entropy, or bounds on the entropy, of shift spaces with overlaps.

\section{Topological entropy of shift spaces with overlap}
\label{sec:entropyshiftwithoverlap}

We are primarily interested in using $(\Sigma,\intgraph)$ to help us obtain a lower bound for the entropy of $f$. Despite the fact that $(\Sigma,\intgraph)$ is not a topological space, we can make a reasonable definition of its topological entropy, and we prove that it is a lower bound for the entropy of $f$ (Theorem \ref{thm:entropybound}).  

If we ignore the intersections, then $\Sigma$ is a shift space and we can compute $\ent(\Sigma)$ by computing the growth rate of the number of words of length $n$ as $n$ goes to infinity (\cite[\S4.1]{LM}).  Let $\blocks{n}(\Sigma)$ be the set of words of length $n$; then the topological entropy of $\Sigma$ is
\[\ent(\Sigma)=\lim_{n\to\infty}\frac{1}{n} \log |\blocks{n}(\Sigma)| .\]  
 
Recall that if the shift space $\Sigma$ is a subshift of finite type with transition graph $\transgraph$, then there is a particularly easy way to compute the entropy. The graph $\transgraph$ has an associated \emph{transition matrix}, $A_{\transgraph}$, in which the $(i,j)$-entry is 1 if there is an edge from the $i$th vertex to the $j$th vertex and is 0 otherwise. Then \[h(\Sigma)=\log(\lambda),\] where $\lambda$ is the Perron eigenvalue of $A_{\transgraph}$---the unique largest real eigenvalue of $A_{\transgraph}$.

In our setting, we want to count the number of words of length $n$ that are mutually distinguishable, which in general is less than $|\blocks{n}(\Sigma)|$. (For example, if vertices 1 and 2 overlap, then we would not be able to distinguish the words $(1,3)$ and $(2,3)$.)  Let $B$ be a set of words in $\Sigma$ of length $n$.  We say that $B$ is \emph{$n$-separated} if no two words in $B$ are indistinguishable.  Let $\sep{n}(\Sigma,\intgraph)$ be an $n$-separated set of maximum cardinality.  We define $\olenti(\Sigma)$, the \emph{topological entropy of a shift space with overlaps}, to be
\[\olenti(\Sigma)= \limsup_{n\to\infty}\frac{1}{n}\log |\sep{n}(\Sigma,\intgraph)|.\]
  
The following theorem provides the motivation for studying the entropy of shifts with overlaps.
\begin{thm}\label{thm:entropybound}
If $f$ is a dynamical realization of $(\Sigma, \intgraph)$, then $\ent(f) \ge \olenti(\Sigma)$.
\end{thm}

\begin{proof}
Let $\ep = \min d(N_i,N_j)$, where the minimum is over all disjoint pairs $N_i$ and $N_j$.  Then an $n$-separated word in $\Sigma$ corresponds to an $(n,\ep)$-separated orbit for $f$. Since $\displaystyle \ent(f) = \lim_{\ep\to0} \limsup_{n\to\infty} \frac1n \log r(n,\ep)$, where $r(n,\ep)$ is the maximum cardinality of a set of $(n,\ep)$-separated orbits (\cite[\S8.1]{Ro}), the result follows. 
\end{proof}

The rest of the paper is devoted to computing or finding lower bounds for $\olenti(\Sigma)$. 

\section{Independent subsets of the intersection graph}
\label{sec:indep}

In this section we take advantage of the fact that if all words are distinguishable, then we can treat our shift space with overlaps as an ordinary shift space.

\begin{prop}
Let $(\Sigma, \intgraph)$ be a shift space with overlaps in which the edge set of $\intgraph$ is empty.  Then $\olenti(\Sigma)=\ent(\Sigma)$.
\end{prop}

\begin{cor}
Let $(\Sigma, \intgraph)$ be a shift space with overlaps. If $\Sigma'\subset\Sigma$ is a shift space in which all elements are pair-wise distinguishable, then $\olenti(\Sigma)\ge \olenti(\Sigma')=\ent(\Sigma')$.
\end{cor}

Let $V(\graph)$ and $E(\graph)$ denote the vertex and edge sets of a graph $\graph$, respectively. Given a (nondirected) graph $\graph$, a subset $V\subset V(\graph)$ is \emph{independent} if there are no edges between any pair of vertices in $V$.  The maximum cardinality of an independent subset of $\graph$ is the \emph{independence number} of $\graph$ and is denoted $\ind(\graph)$.  

For any $V\subset V(\intgraph)$, let $\Sigma_{V}=\{(a_{0},a_{1},\ldots)\in\Sigma: a_{i}\in V\}$. Notice that if $\Sigma$ is a shift space, then $\Sigma_{V}$ is also a shift space. A \emph{vertex-induced subgraph} of a graph $\graph$ is a subgraph $\graph'$ with the property that any $\graph$-edge whose endpoints are in $V(\graph')$ is in $E(\graph')$. If $V\subset V(\intgraph)$, let $\intgraph_{V}$ denote the vertex-induced subgraph of $\intgraph$ with vertex set $V$.

Corollaries \ref{cor:indep} and \ref{cor:sft} give us a tactic for finding a lower bound for $\olenti(\Sigma)$. Remove enough vertices of $\intgraph$ so that we have an independent set of vertices; then the corresponding shift space with overlaps has no overlaps---it is just a shift space. The entropy of this shift space is a lower bound for $\olenti(\Sigma)$.

\begin{cor}\label{cor:indep}
If $V\subset V(\intgraph)$ is an independent set, then $\olenti(\Sigma)\ge \olent{\intgraph_{V}}(\Sigma_{V})=\ent(\Sigma_{V})$.
\end{cor}

In the following corollaries we assume that $\Sigma$ is a shift of finite type with transition graph $\transgraph$ and intersection graph $\intgraph$.

\begin{cor}\label{cor:sft}
Let $V\subset V(\intgraph)$ be an independent set and let $\transgraph'\intgraph'$ be the vertex-induced subgraph of $\transgraph\intgraph$ with vertex set $V$. Then $\olenti(\Sigma)\ge\log(\lambda)$ where $\lambda$ is the Perron eigenvalue of $A_{\transgraph'}$.
\end{cor}

\begin{cor}\label{cor:complete}
Suppose $\transgraph$ is a complete digraph. Then $\olenti(\Sigma)\ge\log(\ind(\intgraph))$. 
\end{cor}
\begin{proof}
Let $V$ be a maximum cardinality independent subset of $\intgraph$. The vertex-induced subgraph of $\transgraph$ with vertex set $V$ is a complete digraph with $\ind(\intgraph)$ vertices. By Corollary \ref{cor:sft}, $\olenti(\Sigma)\ge\log(\ind(\intgraph))$.
\end{proof}

\section{Decomposition into irreducible and primitive components}
\label{sec:decomp}

We begin this section with some basic definitions from graph theory and linear algebra. (See \cite[ch.~4]{LM} or any non-introductory linear algebra text for more details.)

 Let $\graph$ be a directed graph. A {\em vertex path} $(a_0,a_1,\ldots,a_{l-1})$ of length $l$ is a sequence of vertices in $\graph$ such that there is an edge from vertex $a_i$ to $a_{i+1}$ for each $i$, $0\le i\le l-2$. In what follows, we assume that $\graph$ has $n$ vertices and no parallel edges and that every vertex has at least one edge leaving it and one edge entering it.  Then the adjacency matrix for $\graph$, $A_{\graph}$, is an $n\times n$ (0,1)-matrix with no row or column containing all zeros. Notice that for $k>1$ the matrix $A_{\graph}^{k}$ need not be a $(0,1)$-matrix. However, if the $(i,j)$-entry of $A_{\graph}^{k}$ is nonzero, then there exists a vertex path of length $k+1$ from vertex $i$ to vertex $j$ (actually, the value of the $(i,j)$ gives the number of such vertex paths).

A graph is \emph{irreducible} provided there is a vertex path between any two vertices, and a matrix is \emph{irreducible} if it is the adjacency matrix for an irreducible graph. It is always possible to decompose a graph into its irreducible components; the corresponding matrix is (after rearranging the order of the vertices, or equivalently, conjugating by a permutation matrix) in block-triangular form with each of the blocks on the diagonal the adjacency matrix for one of the irreducible components of the graph. 

An irreducible graph $\graph$ may exhibit cyclic behavior. In particular for each vertex $i$, $\Per(i)=\gcd\{n:\text{there is a vertex path of length }n+1\text{ from }i\text{ to itself}\}$ exists and is the same value for all $i$. This common value is called the \emph{period} of $\graph$, denoted $\Per(\graph)$. Define the \emph{period} of an irreducible $(0,1)$-matrix $A$, $\Per(A)$, to be the period of the associated graph. If the period of the graph or the matrix is 1, then we call it \emph{primitive}. If $A$ is an irreducible matrix with $p=\Per(A)$, then (after reordering vertices) $A^{p}$ is a block triangular matrix with primitive matrices along the diagonal. Finally, if $A$ is a primitive matrix, then $A^{k}$ is eventually positive.  The \emph{index of primitivity} of $A$, $\gamma(A)$, is the least integer $k$ such that $A^k>0$.  

Now suppose our shift space with overlaps has an associated graph $\transgraph\intgraph$; that is, we consider $(\Sigma,\intgraph)$ where $\Sigma$ is a subshift of finite type with transition graph $\transgraph$. We make the standing assumption that every vertex of $\transgraph$ has at least one edge leaving it and one edge entering it; if not, we can remove the ``stranded'' vertices without affecting the dynamics. Recall that the graph-theoretical notions of irreducibility and primitivity have dynamical interpretations for the shift: $\transgraph$ is irreducible iff $\Sigma$ is \emph{topologically transitive}, and $\transgraph$ is primitive iff $\Sigma$ is \emph{topologically mixing}.

\begin{thm}\label{thm:primitive}
Suppose $\transgraph$ is primitive with index of primitivity $\gamma$. Then $\displaystyle\olenti(\Sigma)\ge\frac{\log(\ind(\intgraph))}{\gamma}$.
\end{thm}

\begin{proof}

Let $\{b_1,\ldots,b_{\ind(\intgraph)}\}$ be an independent set of vertices.  Since $A^\gamma >0 $, for any pair $b_i$, $b_j$, there is a word $(b_{i},a_2,a_3,\ldots,a_{\gamma},b_j)$ in $\Sigma$.  More generally, for any $n$ and any $b_{i_1},\ldots,b_{i_n}$, there is a word $$(b_{i_1},a_2,\ldots, a_{\gamma},b_{i_2},a_{\gamma+2},\ldots, a_{2\gamma},\ldots ,b_{i_n},a_{(n-1)\gamma+2},\ldots ,a_{n\gamma})$$ in $\Sigma$. The collection of all $\ind(\intgraph)^n$ such words is $n\gamma$-separated, so we have that $|\sep{n\gamma}(\Sigma,\intgraph)| \ge \ind(\intgraph)^n$, and thus
$$\olenti(\Sigma)= \limsup_{n\to\infty}\frac{1}{n}\log |\sep{n}(\Sigma,\intgraph)| \ge  \limsup_{n\to\infty}\frac{1}{n\gamma}\log |\sep{n\gamma}(\Sigma,\intgraph)| \ge \frac{\log(\ind(\intgraph))}{\gamma}.$$
\end{proof}

\begin{cor}
If $\transgraph$ is primitive with $n$ vertices, then $\displaystyle \olenti(\Sigma) \ge \frac{\log (\ind(\intgraph))}{n^{2}-2n+2}$.
\end{cor}

\begin{proof}
This follows from the fact that $\gamma(\transgraph) \le n^{2}-2n+2$ (\cite[Thm.~4.14]{BP}).
\end{proof}

In general $\transgraph$ is not primitive. In this case we break down $\transgraph$ into its irreducible components. Each irreducible component has some period $p$, and can thus be decomposed into $p$ primitive components. We then apply Theorem \ref{thm:primitive} to obtain the following corollary. 

\begin{cor}
Suppose $\transgraph'\subset\transgraph$ is a primitive component with index of primitivity $\gamma$ that resides in an irreducible component of period $p$ and that $\intgraph'$ is the associated intersection graph. Then \[\olenti(\Sigma)\ge\frac{\log(\ind(\intgraph'))}{p\gamma}.\]
\end{cor}

\begin{cor}
If there are two vertices of $\transgraph$ that are in the same primitive component and are not joined by an $\intgraph$-edge, then $(\Sigma,\intgraph)$ has positive entropy.
\end{cor}

\begin{ex}\label{ex:bigexample}
Let $U\subset\R^{2}$ be an open set and $f:U\to\R^{2}$ be a continuous function that maps the sets $N_{1},\ldots,N_{11}\subset U$ as shown in Figure \ref{fig:bigexample}. The associated graph $\transgraph\intgraph$ is shown on the right.

\begin{figure}[ht]
\includegraphics[width=5in]{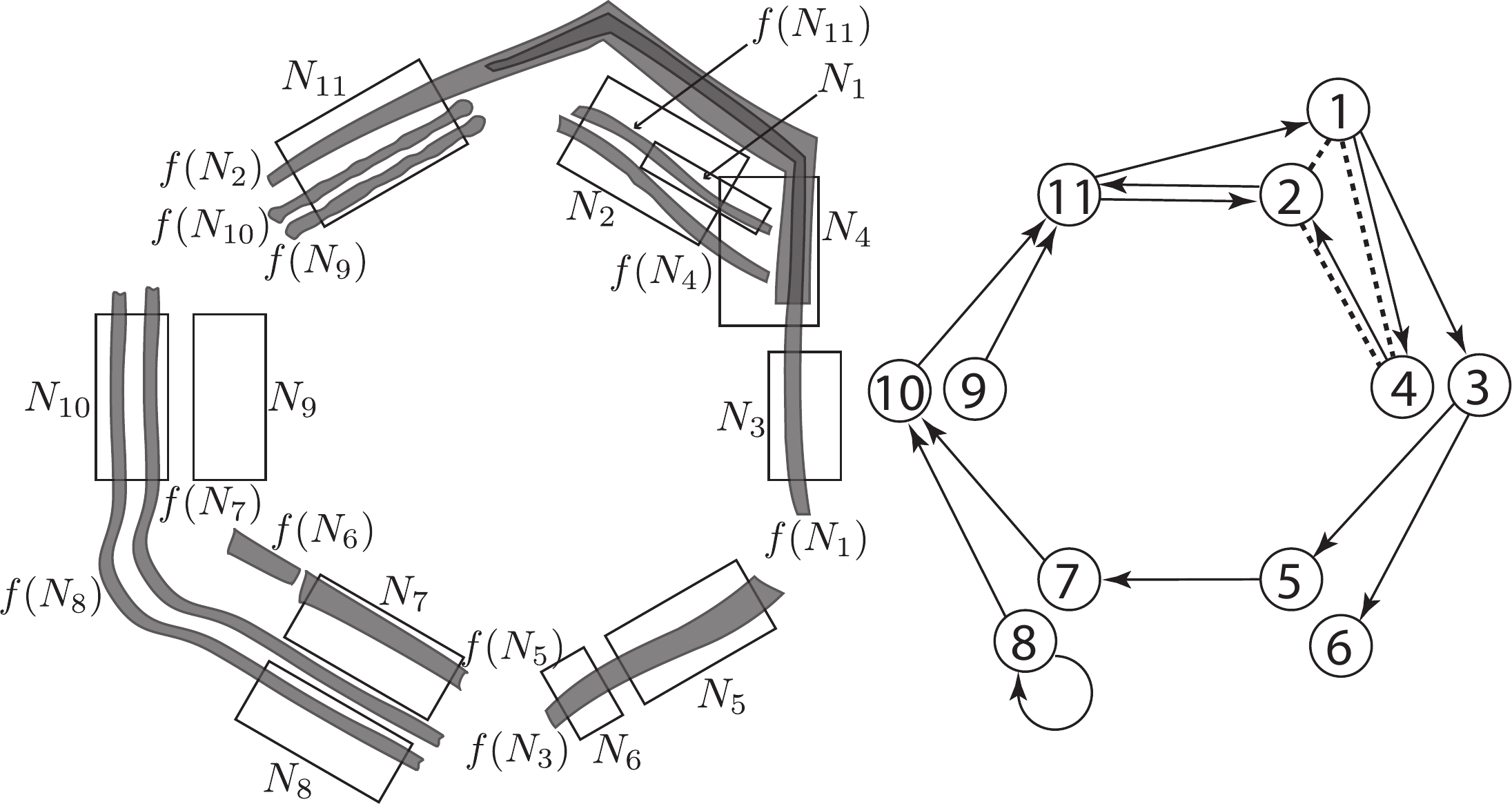}
\caption{}\label{fig:bigexample}
\end{figure}

First we remove vertices $6$ and  $9$ since they lack an outgoing and an incoming edge in $\transgraph$, respectively. (Note: if we also removed two of the vertices $1$, $2$, and $4$ to obtain an independent set, then we could apply Corollary \ref{cor:sft}, but the corresponding shift space would have zero entropy.) The resulting transition graph has two irreducible components, with vertex sets $\{8\}$ and $\{1,2,3,4,5,7,9,10,11\}$. The second of these has period 2 and it decomposes into the primitive components  $\{3,4,7,11\}$ and $\{1,2,5,10\}$. It is straightforward to show that both of these have index of primitivity $4$ and that the first of these is an independent set. Thus we conclude that $\ent(f)\ge\olenti(\Sigma)\ge \frac{1}{2\cdot 4}\log(4)\approx0.173$. We will return to this example later.
\end{ex}

\begin{ex}\label{ex:doubling}
Let $S^1$ be the circle, viewed as $\R/\Z$ and let $f:S^{1}\to S^{1}$ be a map that is $C^{0}$-close to the doubling map ($x\mapsto 2x$).  Let $N_1=[-0.1,0.35]$, $N_2=[0.15,0.6]$, $N_3=[0.4,0.85]$, and $N_4=[0.65,1.1]$.  The transition matrix and the graph $\transgraph\intgraph$ are shown in Figure \ref{fig:jimsexample}.

\begin{figure}[ht]
\includegraphics[width=4in]{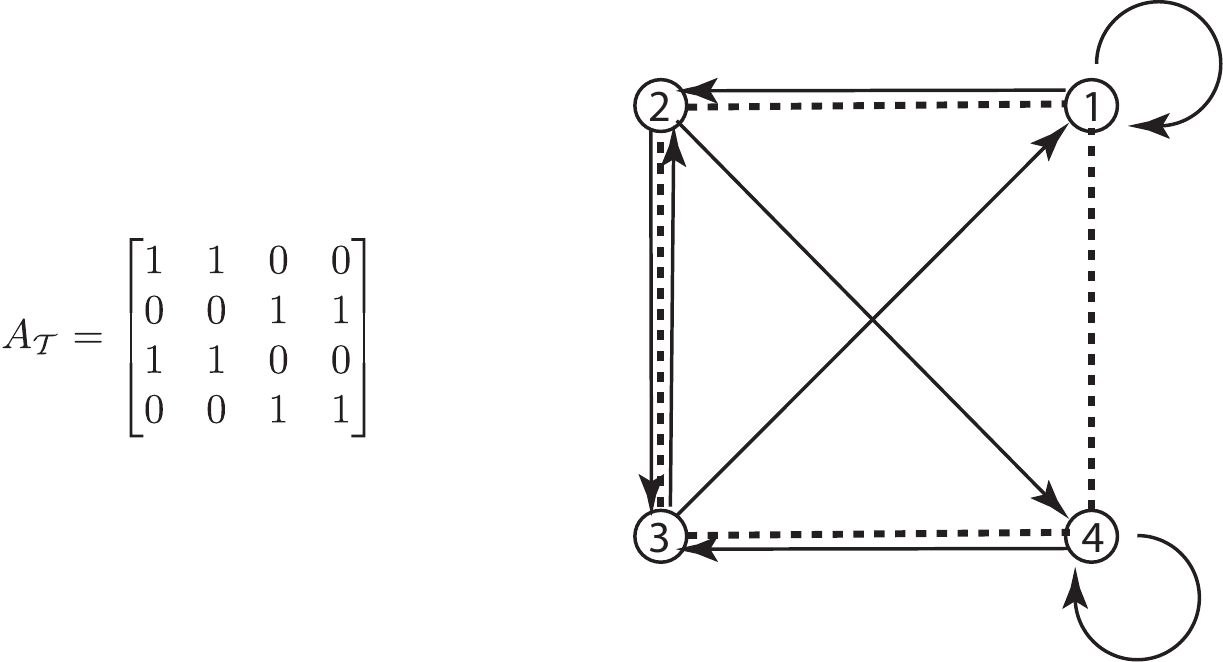}
\caption{}\label{fig:jimsexample}
\end{figure}
Since $A_{\transgraph}^2$ consists of all 1's, $A_{\transgraph}$ is irreducible with $\Per(A_{\transgraph})=1$ and primitive with $\gamma(A_{\transgraph})=2$.  The maximal independent sets in $\intgraph$ are $\{1,3\}$ and $\{2,4\}$, so $\ind(\intgraph)=2$.  Thus $\ent(f)\ge\olenti(\Sigma)\ge\frac{1}{2}\log(2)$, which is half the entropy of the doubling map. We will return to this example.
\end{ex}

\section{Sofic shifts and $\intgraph$-component shifts}
\label{sec:sofic}

Vertices that are adjacent in $\intgraph$ are indistinguishable. This is a problem for our symbolic dynamics, since different symbols may not correspond to different points. One way of working around this problem is to give neighboring vertices the same label. In other words, give each vertex in a connected component of $\intgraph$ the same label. 

Suppose $\intgraph$ has connected components $A_{1},\ldots,A_{r}$. Then create the shift space $\Sigma_{\intgraph}$ on the alphabet $\{A_{1},\ldots,A_{r}\}$ as follows. The element $(A_{i_{0}},A_{i_{1}},\ldots)\in\Sigma_{\intgraph}$ if and only if there is an element $(a_{0},a_{1},\ldots)\in\Sigma$ such that $a_{k}\in A_{i_{k}}$ for all $k$. We call $\Sigma_{\intgraph}$ the \emph{$\intgraph$-component shift space} for $(\Sigma,\intgraph)$. 

\begin{thm}\label{thm:componentshift}
Let $\Sigma_{\intgraph}$ be the $\intgraph$-component shift space associated to $(\Sigma,\intgraph)$. Then $\olenti(\Sigma)\ge \ent(\Sigma_{\intgraph})$.
\end{thm}
\begin{proof}
Let $\blocks{m}(\Sigma_{\intgraph})$ be the set of words of length $m$ in $\Sigma_{\intgraph}$. For each element $(A_{i_{0}},\ldots,A_{i_{m-1}})\in\blocks{m}(\Sigma_{\intgraph})$, pick one word $(a_{0},\ldots,a_{m-1})$ in $\Sigma$ such that $a_{k}\in A_{i_{k}}$ for $k=0,\ldots,m-1$. The collection of these words in $\Sigma$ are $m$-separated. However, it may not be a maximal $m$-separated set. Thus $|\sep{m}(\Sigma,\intgraph)|\ge|\blocks{m}(\Sigma_{\intgraph})|$, and hence $\olenti(\Sigma)\ge \ent(\Sigma_{\intgraph})$.
\end{proof}

A set of vertices in a nondirected graph form a \emph{clique} if every pair of vertices in the set are joined by an edge.

\begin{prop}
\label{prop:cliques}
If the vertices in each connected component of $\intgraph$ form a clique, then $\olenti(\Sigma)=\ent(\Sigma_{\intgraph})$.
\end{prop}
\begin{proof}
By Theorem \ref{thm:componentshift}, all we must prove is that $\olenti(\Sigma)\le\ent(\Sigma_{\intgraph})$. Let $\sep{m}(\Sigma,\intgraph)$ be a maximal $m$-separated set. Consider the function $\psi:\sep{m}(\Sigma,\intgraph)\to\blocks{m}(\Sigma_{\intgraph})$ given by $\psi(a_{0},a_{1},\ldots,a_{m-1})=(A_{i_{0}},A_{i_{1}},\ldots,A_{i_{m-1}})$ where $a_{k}\in A_{i_{k}}$ for all $k$. We will prove that this function is injective. Suppose $\psi(a_{0},a_{1},\ldots,a_{m-1})=\psi(b_{0},b_{1},\ldots,b_{m-1})$ for some $(a_{0},a_{1},\ldots,a_{m-1})$, $(b_{0},b_{1},\ldots,b_{m-1})\in\sep{m}(\Sigma,\intgraph)$. Then $a_{k},b_{k}\in A_{i_{k}}$ for all $k$. But the subgraphs of $\intgraph$ with vertex sets $A_{i_{k}}$ are cliques, so $a_{k}$ and $b_{k}$ are the same or are indistinguishable. Thus $(a_{0},a_{1},\ldots,a_{m-1})$ and $(b_{0},b_{1},\ldots,b_{m-1})$ are indistinguishable. Since $\sep{m}(\Sigma,\intgraph)$ is $m$-separated, $(a_{0},a_{1},\ldots,a_{m-1})=(b_{0},b_{1},\ldots,b_{m-1})$, and $\psi$ is injective. Because $\psi$ is an injective function between finite sets, we conclude that $|\sep{m}(\Sigma,\intgraph)|\le|\blocks{m}(\Sigma_{\intgraph})|$. Hence  $\olenti(\Sigma)\le\ent(\Sigma_{\intgraph})$.
\end{proof}

An important special case occurs when $\Sigma$ is a subshift of finite type. Then $\Sigma_{\intgraph}$ is a sofic shift (see \cite[Ch.~3]{LM} for more about sofic shifts and the terms in this paragraph). If $(\Sigma,\intgraph)$ is represented by the graph $\transgraph\intgraph$, then $\Sigma_{\intgraph}$ is the \emph{sofic shift} represented by the directed graph obtained from $\transgraph$ by giving all vertices in a connected component of $\intgraph$ the same label. If this graph is \emph{right-resolving} (that is, if every edge from a given vertex points to a different label), then the entropy of the sofic shift is simply the log of the Perron eigenvalue of the adjacency matrix for $\transgraph$. If the graph is not right-resolving then there is a mechanical procedure for transforming it  into a right-resolving graph representing the same sofic shift. In particular, the entropy is easy to compute.

\begin{ex}\label{ex:goldenmean}
Consider the graph $\transgraph\intgraph$ shown on the left in Figure \ref{fig:goldenmeanshift}. If we give vertices 1 and 3 the same label we obtain the ``golden mean'' shift shown on the right (a shift that is well-known to be sofic but not of finite type). By Proposition \ref{prop:cliques} the entropies of the two are equal, and since the graph for the sofic shift is right-resolving, we can easily compute the entropy from the adjacency matrix: $\olenti(\Sigma)=\ent(\Sigma_{\intgraph})=\log(\frac{1}{2}(1+\sqrt{5}))$.
\begin{figure}[ht]
\includegraphics[width=4in]{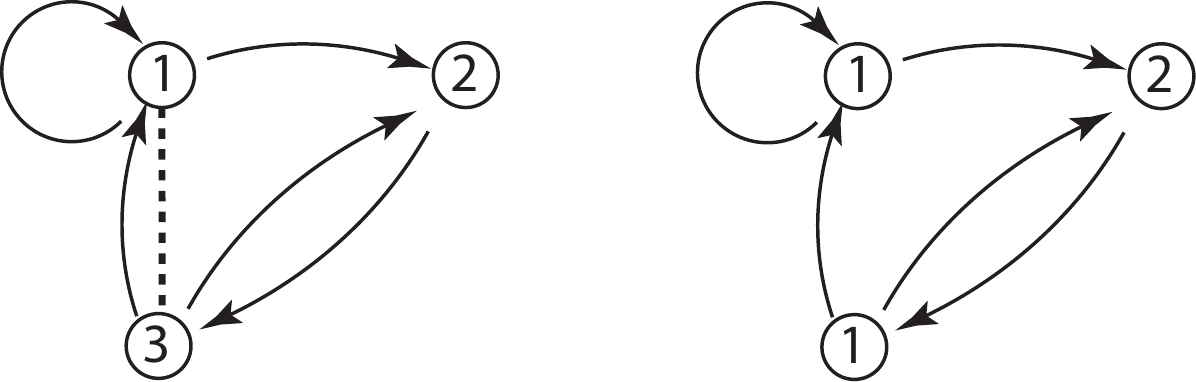}\\
\caption{}\label{fig:goldenmeanshift}
\end{figure}
\end{ex}

\begin{ex}\label{ex:bigexample2}
Consider the sofic shift associated to the shift space with overlaps in Example \ref{ex:bigexample} (with vertices 6, 8, and 9 removed). That is, we give vertices 1, 2, and 4 the same label (the graph on the left in Figure \ref{fig:bigsofic}). The graph is not right-resolving, but after being put in right-resolving form (the graph on the right) we find that $\olenti(\Sigma)\ge\log\big(\frac{1}{6}(108+12\sqrt{69})^{1/3}+2(108+12\sqrt{69})^{-1/3}\big))\approx .281$, which is a larger lower bound than the one obtained in Example \ref{ex:bigexample}. 
\begin{figure}[ht]
\includegraphics[width=4.5in]{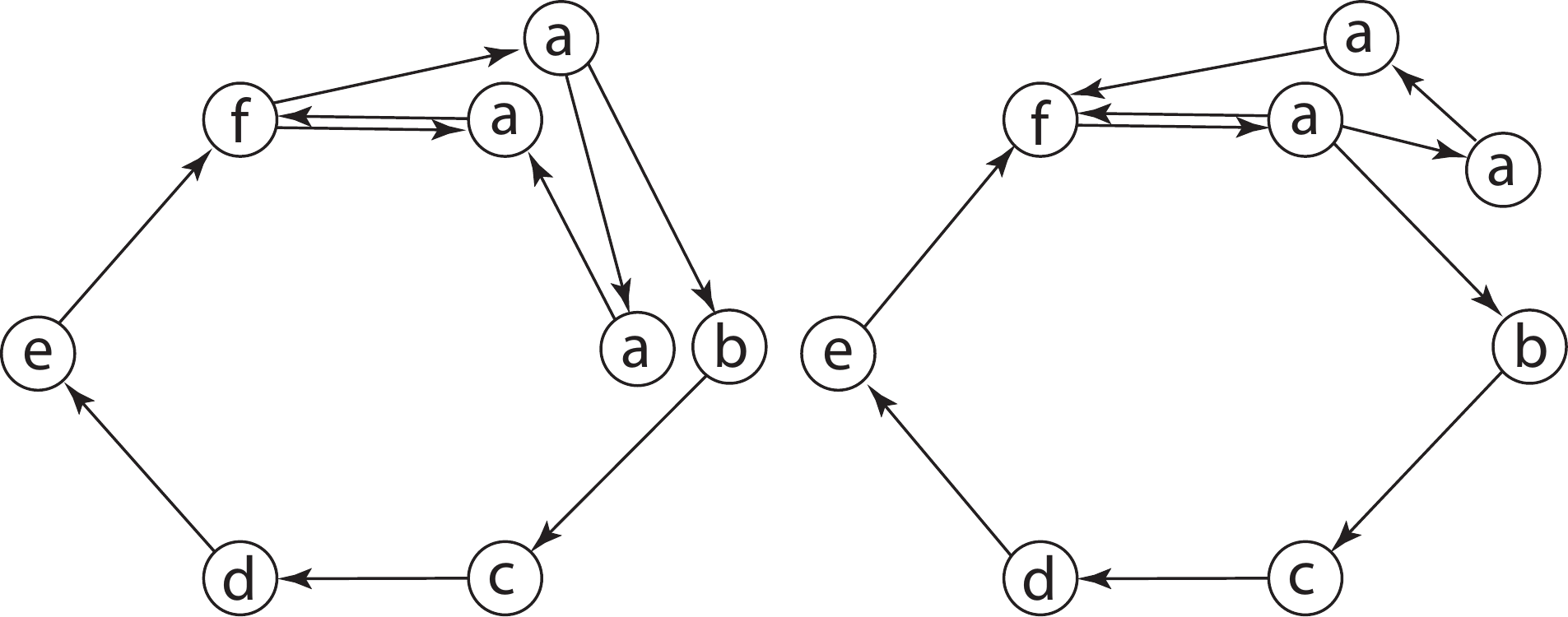}
\caption{}\label{fig:bigsofic}
\end{figure}

\end{ex}

\begin{ex}\label{ex:doubling2}
On the other hand, the sofic shift corresponding to the doubling map of Example \ref{ex:doubling} has only one symbol, so we obtain a lower bound of zero for the entropy.
\end{ex}

In some circumstances we could obtain a larger lower bound for the entropy by using a combination of removing vertices (as described in Section \ref{sec:indep}) and relabeling (as described in this section) than we could using either of these techniques alone. We could remove a few vertices to increase the number of connected components of $\intgraph$, then relabel  to obtain an $\intgraph$-component shift with high entropy.

\section{Higher shifts} \label{sec:higher}
In this section we examine higher shifts. This is a well-known way of altering the transition graph for a subshift of finite type to obtain a new graph which generates a topologically conjugate subshift of finite type. The higher vertex shift allows us to subdivide the set of vertices, thus presumably giving us a greater number of disjoint vertices. We show how to implement this notion for shift spaces with overlap generated from a graph $\transgraph\intgraph$. We refer the reader to \cite[Sect.~2.3]{LM} for details on higher shifts for edge-labeled graphs; the constructions can be modified easily for vertex-labeled graphs, which is what we need. 

For $m\ge1$ we define the \emph{$m$th higher vertex graph} $\transgraph_{[m]}$ to have vertex set equal to the collection of all vertex paths of length $m$ in $\transgraph$, with an edge from a vertex $(i_0,\dots,i_{m-1})$ to a vertex $(i_1,\dots,i_{m-1},i_m)$ provided there is an edge in $\transgraph$ from $i_{m-1}$ to $i_m$.  Note that the first higher vertex shift $\transgraph_{[1]}$ is simply $\transgraph$.

The intersection graph $\intgraph$ induces an intersection graph $\intgraph_{[m]}$ for the vertices of $\transgraph_{[m]}$.  There is an (undirected) edge between the vertices $(i_0,\dots,i_{m-1})$ and $(j_0,\dots,j_{m-1})$ if the words are indistinguishable; that is, if there is an edge in $\intgraph$ between $i_k$ and $j_k$ for all $k$.

\begin{thm}\label{thm:higherentropy}
If  $\transgraph_{[m]}\intgraph_{[m]}$ is the graph associated to the $m$th higher shift of $(\Sigma_{\transgraph},\intgraph$), then $\olent{\intgraph_{[m]}}(\Sigma_{\transgraph_{[m]}})=\olenti(\Sigma_{\transgraph})$.
\end{thm}
\begin{proof}
Let $m\ge 1$ be fixed. For each $n$, there exists a natural bijective function $\psi_{n}:\blocks{n+m-1}(\Sigma_{\transgraph})\to\blocks{n}(\Sigma_{\transgraph_{[m]}})$. Specifically, $\psi_{n}(b_{0},\ldots,b_{n+m-2})=(w_{0},\ldots,w_{n-1})$, where $w_{i}=(b_{i},\ldots,b_{m+i-1})$. Moreover, $\psi_{n}(b_{0},\ldots,b_{n-1})$ and $\psi_{n}(c_{0},\ldots,c_{n-1})$ are distinguishable if and only if $(b_{0},\ldots,b_{n-1})$ and $(c_{0},\ldots,c_{n-1})$ are distinguishable. Thus,  $|\sep{n+m-1}(\Sigma_{\transgraph},\intgraph)|=|\sep{n}(\Sigma_{\transgraph_{[m]}},\intgraph_{[m]})|$, and hence
\begin{align*}
\olent{\intgraph_{[m]}}(\Sigma_{\transgraph_{[m]}})&= \limsup_{n\to\infty}\frac{1}{n}\log |\sep{n}(\Sigma_{\transgraph_{[m]}},\intgraph_{[m]})|\\
&= \limsup_{n\to\infty}\frac{1}{n}\log |\sep{n+m-1}(\Sigma_{\transgraph},\intgraph)|\\
&=\olenti(\Sigma_{\transgraph}).
\end{align*}  
\end{proof}

\begin{prop}
$\ind(\intgraph_{[m]})\ge\ind(\intgraph)$.
\end{prop}

\begin{proof}
(Recall that we are assuming that each vertex has at least one edge leaving it.) If $\{b_1,\ldots,b_{\ind(\intgraph)}\}$ is an independent set of vertices of $\intgraph$, and $w_i=(b_i,a_{i_1},\ldots,a_{i_{m-1}})$ is, for each $b_i$, any word of length $m$, then $\{w_1,\ldots,w_{\ind(\intgraph)}\}$ is an independent set of vertices of $\intgraph_{[m]}$.

\end{proof}

\begin{prop}\label{prop:higherexp}
Assume that $\transgraph$ has at least two vertices.  If $\transgraph$ is primitive, then so is $\transgraph_{[m]}$, and $\gamma(\transgraph_{[m]})=\gamma(\transgraph)-1+m$.
(If $\transgraph$ has only one vertex, then $\transgraph_{[m]}$ is isomorphic to $\transgraph$.)
\end{prop}

\begin{proof}
A graph is primitive with exponent $\gamma$ if and only if from any vertex to any other there is a vertex path of length $\gamma+1$, and $\gamma$ is the smallest number with this property.  
Observe that $(i_{m-1},k_1,\dots,k_p,j_0)$ is a vertex path in $\transgraph$ (of length $p+2$, with $p\ge 0$) if and only if 
\[((i_0,\dots,i_{m-1}), (i_1,\dots,i_{m-1},k_1),\dots, (k_p,j_0,\dots,j_{m-2}), (j_0,\dots,j_{m-1}))\]
is a vertex path in $\transgraph_{[m]}$ (of length $p+m+1$) for any vertices $(i_0,\dots,i_{m-1})$ and $(j_0,\dots,j_{m-1})$.  Between any two vertices of $\transgraph$ there is a vertex path of length $p+2=\gamma(\transgraph)$, so between any two vertices of $\transgraph_{[m]}$ there is a vertex path of length $p+m+1=\gamma(\transgraph)-1+m$.  Thus $\gamma(\transgraph_{[m]})\le \gamma(\transgraph)-1+m$.

To prove the opposite inequality, first assume that $\gamma(\transgraph)>1$.  Then there exist two vertices of $\transgraph$ such that there is no vertex path from the first to the second of length $\gamma(\transgraph)$, and thus two vertices of $\transgraph_{[m]}$ with no vertex path of length $\gamma(\transgraph)-1+m$ between them, so $\gamma(\transgraph_{[m]})\ge \gamma(\transgraph)-1+m$.

Finally, if $\gamma(\transgraph)=1$, then any vertex of $\transgraph$ can follow any other vertex.  For $i\ne j$, it is clear that the shortest vertex path in $\transgraph_{[m]}$ from $(i,\dots,i)$ to $(j,\dots,j)$ has length $m+1$, so $\gamma(\transgraph_{[m]}) \ge m = \gamma(\transgraph)-1+m$.

\end{proof}

\begin{ex}\label{ex:doubling3}
We return to Example~\ref{ex:doubling}.  The vertices of $\transgraph_{[2]}$ are $(1,1)$, $(1,2)$, $(2,3)$, $(2,4)$, $(3,1)$, $(3,2)$, $(4,3)$, and $(4,4)$---one for each edge (or vertex path of length 2) in $\transgraph$. There are nine edges in $\intgraph_{[2]}$. They can be seen in the graph $\transgraph_{[2]}\intgraph_{[2]}$ in Figure \ref{fig:higherexample2}.
\begin{figure}[ht]
\includegraphics[width=4in]{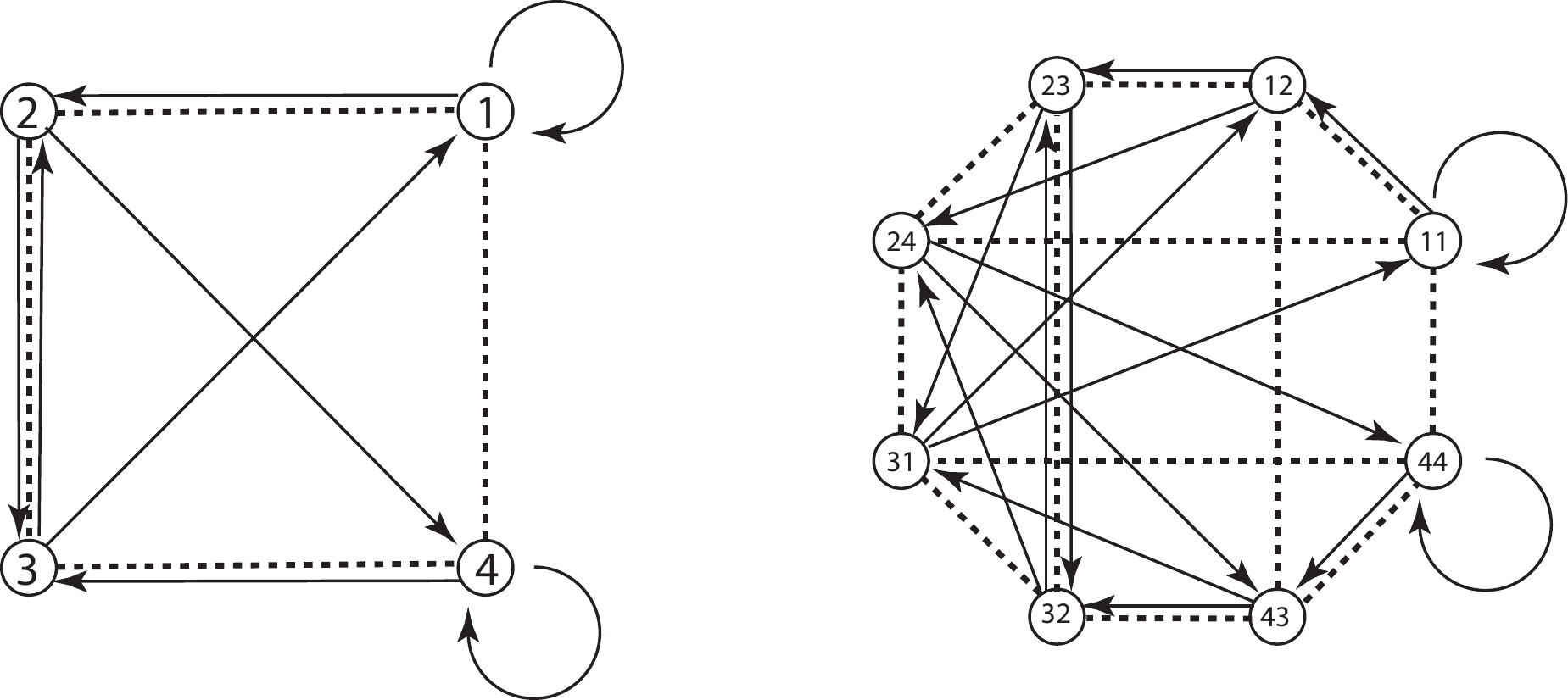}
\caption{}\label{fig:higherexample2}
\end{figure}

Since $\gamma(\transgraph)=2$, Proposition~\ref{prop:higherexp} says that $\gamma(\transgraph_{[2]})=3$, as we can check by observing that $A_{[2]}^2$ has zero entries while $A_{[2]}^3$ does not (where $A_{[2]}$ is the adjacency matrix for $\transgraph_{[2]}$).

Notice that the set of vertices $\{(1,1), (2,3), (3,1), (4,3)\}$ is a maximum cardinality independent set in $\intgraph_{[2]}$, so $\ind(\intgraph_{[2]})=4$. Theorems~\ref{thm:entropybound}, \ref{thm:primitive}, and \ref{thm:higherentropy} tells us that \[\ent(f) \ge \olenti(\Sigma_{\transgraph})=\olent{\intgraph_{[2]}}(\Sigma_{\transgraph_{[2]}})\ge\frac{\log 4}3 = \frac{2\log2}{3},\] which is greater than $\frac{1}{2}\log2$, the lower bound we obtained in Example~\ref{ex:doubling} using $\transgraph$. In this particular example we see that $\gamma(\transgraph_{[m]})=m+1$ and $\ind(\intgraph_{[m]})=2^{m}$. So \[\ent(f)\ge \frac{\log2^{m}}{m+1}=\frac{m\log2}{m+1}\] for all $m$. This implies that $\ent(f)\ge\log2$, the entropy of the doubling map.
\end{ex}

\section{Entropy as a limit}
\label{sec:entropytheorem}

In Example \ref{ex:doubling3} we see that we obtain a very good entropy bound by looking at the sequence of higher shifts. This is true for more than this one example. In this section we show that when the shift space with overlaps is given by a graph $\transgraph\intgraph$, the sequence of higher shifts is closely related to the entropy. In particular, if $\transgraph$ is primitive, the entropy can be expressed as a limit. 

As in Section \ref{sec:higher}, let $\transgraph_{[m]}\intgraph_{[m]}$ denote the graph for the $m$th higher shift. In this case $\blocks{m}(\Sigma)=V(\transgraph_{[m]}\intgraph_{[m]})$. Moreover, an $m$-separated set $B$ in $\Sigma$ corresponds to an independent set in $\intgraph_{[m]}$, and hence $|\sep{m}(\Sigma,\intgraph)|=\ind(\intgraph_{[m]})$.  Thus $\olenti(\Sigma)$ is the growth rate of the independence number of $\intgraph_{[m]}$ as $m$ goes to infinity.

\begin{thm}\label{thm:limit}
If $(\Sigma,\intgraph)$ is the shift space with overlaps associated to the graph $\transgraph\intgraph$, then 
\[\olenti(\Sigma)=\limsup_{m\to \infty}\frac{\log(\ind(\intgraph_{[m]}))}{m}.\]  If $\transgraph$ is primitive, then
\[\olenti(\Sigma) =  \lim_{m\to \infty}\frac{\log(\ind(\intgraph_{[m]}))}{\gamma(\transgraph_{[m]})} = \lim_{m\to\infty}\frac{\log(\ind(\intgraph_{[m]}))}{m}.\]
(In particular, both of these limits exist.)
\end{thm}

\begin{proof}
The limit superior equality follows from the definition of entropy and the discussion preceding this theorem. Suppose $\transgraph$ is primitive with index of primitivity $\gamma(\transgraph)$. Then $\transgraph_{[m]}$ is also primitive and $\gamma(\transgraph_{[m]})=\gamma(\transgraph)+m-1$. So it follows that 

\[\olenti(\Sigma) = \limsup_{m\to\infty}\frac{\log(\ind(\intgraph_{[m]}))}{m} =  \limsup_{m\to \infty}\frac{\log(\ind(\intgraph_{[m]}))}{\gamma(\transgraph_{[m]})}\]
Clearly \[\limsup_{m\to\infty}\frac{\log(\ind(\intgraph_{[m]}))}{m} \ge \liminf_{m\to\infty}\frac{\log(\ind(\intgraph_{[m]}))}{m} = \liminf_{m\to \infty}\frac{\log(\ind(\intgraph_{[m]}))}{\gamma(\transgraph_{[m]})}.\]
So to complete the proof it is enough to show that \[\liminf_{m\to\infty}\frac{\log(\ind(\intgraph_{[m]}))}{m} \ge 
\olenti(\Sigma).\]

Since we have $\ind(\intgraph_{[m]})$ $m$-separated words and $\transgraph$ is primitive, we can make from them $(\ind(\intgraph_{[m]}))^{\lfloor M/\gamma(\transgraph_{[m]})\rfloor}$ $M$-separated words by concatenating, as in the proof of Theorem~\ref{thm:primitive}.  Thus
$$\ind(\intgraph_{[M]}) \ge  \ind(\intgraph_{[m]})^{\lfloor M/\gamma(\transgraph_{[m]})\rfloor},$$ and therefore $$\liminf_{M\to\infty}\frac{\log(\ind(\intgraph_{[M]}))}{M} \ge \frac{\log(\ind(\intgraph_{[m]}))}{\gamma(\transgraph_{[m]})}.$$  Since this holds for all $m$, we have that 
$$\liminf_{M\to\infty} \frac{\log(\ind(\intgraph_{[M]}))}{M} \ge \limsup_{m\to\infty} \frac{\log(\ind(\intgraph_{[m]}))}{\gamma(\transgraph_{[m]})} 
 =  \olenti(\Sigma). $$
\end{proof}

\begin{cor}\label{cor:sup}
If $\transgraph$ is primitive, then
\[\olenti(\Sigma) = \sup_{m>0}\Big\{\frac{\log(\ind(\intgraph_{[m]}))}{\gamma(\transgraph_{[m]})}\Big\}.\]
\end{cor}

\begin{proof}
It is obvious that \[\olenti(\Sigma) \le \sup_{m>0}\Big\{\frac{\log(\ind(\intgraph_{[m]}))}{\gamma(\transgraph_{[m]})}\Big\}.\] We must prove the reverse inequality.  To do this, observe that since we have $\ind(\intgraph_{[m]})$ $m$-separated words and $\transgraph$ is primitive, we can make them into $\ind(\intgraph_{[m]})^k$ words of length $k\gamma(\transgraph_{[m]})$ (for any $k$) by concatenating.  Thus $\ind(\intgraph_{[k\gamma(\transgraph_{[m]})]}) \ge \ind(\intgraph_{[m]})^k$, so 
\[ \frac{\log(\ind(\intgraph_{[k\gamma(\transgraph_{[m]})]})}{k\gamma(\transgraph_{[m]})} \ge \frac{\log((\ind(\intgraph_{[m]}))^k)}{k\gamma(\transgraph_{[m]})} = \frac{\log(\ind(\intgraph_{[m]}))}{\gamma(\transgraph_{[m]})}. 
\]
By taking the limit of both sides as $k\to\infty$, we obtain \[\olenti(\Sigma) \ge \frac{\log(\ind(\intgraph_{[m]}))}{\gamma(\transgraph_{[m]})};\] since this holds for any $m$, we are done.
\end{proof}

\begin{ex}
Recall from Example \ref{ex:doubling3} that $\transgraph_{[m]}$ is a primitive graph with index of primitivity $m+1$ and independence number $2^{m}$. By Corollary \ref{cor:sup}, \[h(f)\ge\olenti(\Sigma) = \sup\Big\{\frac{\log(\ind(\intgraph_{[m]}))}{\gamma(\transgraph_{[m]})}\Big\}=\sup\Big\{\frac{\log 2^{m}}{m+1}\Big\}=\sup\Big\{\frac{m\log 2}{m+1}\Big\}=\ln 2.\] Theorem \ref{thm:limit} gives us the same bound, \[h(f)\ge\olenti(\Sigma)=\lim_{m\to\infty}\frac{\log(\ind(\intgraph_{[m]}))}{m}=\lim_{m\to \infty}\frac{\log 2^{m}}{m}=\ln 2,\] and happens to be the limit of a constant sequence.

\end{ex}

The following example shows that the sequence $\displaystyle \Big\{\frac{\log(\ind(\intgraph_{[m]}))}{\gamma(\transgraph_{[m]})}\Big\}$ may not be nondecreasing.

\begin{ex}\label{ex:notnondecreasing}
Consider the graph $\transgraph\intgraph$ shown in Figure \ref{fig:highercounterexample}.
\begin{figure}[ht]
\includegraphics[width=2in]{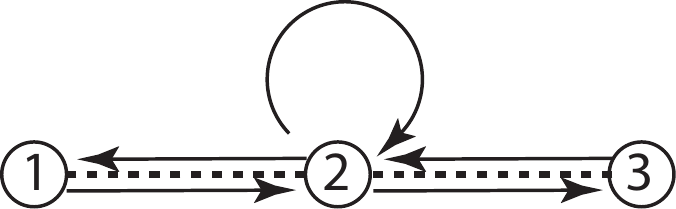}
\caption{}\label{fig:highercounterexample}
\end{figure}

It is easy to check by hand that \[\frac{\log(\ind(\intgraph_{[1]}))}{\gamma(\transgraph_{[1]})}=\frac{1}{2}\ln(2)>\frac{1}{3}\ln(2)=\frac{\log(\ind(\intgraph_{[2]}))}{\gamma(\transgraph_{[2]})}.\] Thus the sequence is not nondecreasing.

We used a computer to find the first several terms of the sequence; based on this, the pattern appears to be $\displaystyle\Big\{\frac{1}{n+1}\ln(2^{m})\Big\}_{n=1}^{\infty}$ where $m=n/2$ for $n$ even and $m=(n+1)/2$ for $n$ odd. Simplified, the $n$th term becomes  $\displaystyle\frac{n\ln 2}{2(n+1)}$ when $n$ is even and $\frac{1}{2}\ln(2)$ when $n$ is odd. The two values given by the theorem (the supremum and the limit superior) are equal and the limit superior is equal to the limit. In particular, this common value, $\frac{1}{2}\ln(2)$ was the first term in the sequence!
\end{ex}

\bibliographystyle{amsplain} 
\bibliography{symbolicindex} 

\providecommand{\bysame}{\leavevmode\hbox to3em{\hrulefill}\thinspace}
\providecommand{\MR}{\relax\ifhmode\unskip\space\fi MR }
\providecommand{\MRhref}[2]{%
  \href{http://www.ams.org/mathscinet-getitem?mr=#1}{#2}
}
\providecommand{\href}[2]{#2}
\begin{thebibliography}{1}

\bibitem{BP}
A.~Berman and R.~J. Plemmons, \emph{Nonnegative matrices in the mathematical
  sciences}, Classics in Applied Mathematics, vol.~9, Society for Industrial
  and Applied Mathematics (SIAM), Philadelphia, PA, 1994, Revised reprint of
  the 1979 original. \MR{MR1298430 (95e:15013)}

\bibitem{KwaCocyclic}
J.~Kwapisz, \emph{Cocyclic subshifts}, Math. Z. \textbf{234} (2000), no.~2,
  255--290. \MR{MR1765882 (2001j:37025)}

\bibitem{KwaTransfer}
\bysame, \emph{Transfer operator, topological entropy and maximal measure for
  cocyclic subshifts}, Ergodic Theory Dynam. Systems \textbf{24} (2004), no.~4,
  1173--1197. \MR{MR2085908 (2005d:37018)}

\bibitem{LM}
D.~Lind and B.~Marcus, \emph{An introduction to symbolic dynamics and coding},
  Cambridge University Press, Cambridge, 1995. \MR{MR1369092 (97a:58050)}

\bibitem{RW7}
D.~Richeson and J.~Wiseman, \emph{Symbolic dynamics for nonhyperbolic systems},
  Proc. Amer. Math. Soc. \textbf{138} (2010), 4373--4385.

\bibitem{Ro}
C.~Robinson, \emph{Dynamical systems}, Studies in Advanced Mathematics, CRC
  Press, Boca Raton, FL, 1995. \MR{97e:58064}

\bibitem{soss}
A.~B. Sossinsky, \emph{Tolerance space theory and some applications}, Acta
  Appl. Math. \textbf{5} (1986), no.~2, 137--167. \MR{823824 (87d:92070)}

\bibitem{szymczak-decomp}
A.~Szymczak, \emph{The {C}onley index for decompositions of isolated invariant
  sets}, Fund. Math. \textbf{148} (1995), no.~1, 71--90. \MR{96m:58154}

\end{thebibliography}
\end{document}